\documentclass[a4paper,12pt]{article}
\usepackage[ansinew]{inputenc}
\usepackage{amsfonts}
\usepackage{latexsym}
\usepackage{amsmath}
\usepackage{amssymb}
\usepackage{graphicx}
\usepackage{color}
\newcommand{\C}{\mathbb{C}}
\newcommand{\N}{\mathbb{N}}
\newcommand{\Z}{\mathbb{Z}}

\newtheorem{proposition}{Proposition}[section]

\newtheorem{defin}{Definition}[section]

\newtheorem{theorem}[defin]{Theorem}
\newtheorem{exa}{Example}[section]
\newenvironment{example}{\begin{exa}\rm}{\end{exa}}
\newtheorem{exas}{Examples}

\newtheorem{lemma}[defin]{Lemma}

\newtheorem{corollary}[defin]{Corollary}
\newenvironment{proof}
{\noindent{\it Proof.}}{\hfill $\Box$\par\vspace{2.5mm}}
\newtheorem{remark}[defin]{Remark}

\newcommand{\abs[1]}{\lvert#1\rvert}

\numberwithin{equation}{section}

\title{Difference ``abc" theorem for entire functions
 and Difference analogue of truncated version of Nevanlinna second main theorem}
\vspace{0.5cm}

\author{Rui-Chun Chen~
 and Zhi-Tao Wen\footnote{Wen is the corresponding author and supported  by the National Natural Science Foundation of China (No.~12471076) and LKSF STU-GTIIT Joint-research Grant (No. 2024LKSFG06).
} }
\date{}

\begin{document}
\maketitle

\begin{abstract}
In this paper, we focus on the difference analogue of the Stothers-Mason theorem for entire functions of order less than 1, which can be seen as difference $abc$ theorem for entire functions. We also obtain the difference analogue of truncated version of Nevanlinna second main theorem which reveals that a subnormal meromorphic function $f(z)$ such that $\Delta f(z)\not\equiv 0$ cannot have too many points with long length in the complex plane. Both theorems depend on new definitions of the length of poles and zeros of a given meromorphic function in a domain. As for the application, we consider entire solutions of difference Fermat functional equations.

\medskip
\noindent
\textbf{Keyword}: length of zeros, length of poles, Stothers--Mason theorem, subnormal.

\medskip
\noindent
\textbf{2020MSC}: 30D35, 30D20
\end{abstract}

\section{Introduction}
Let $a,b,c$ be nonzero relatively prime integers such that $a+b=c$,
and let $\text{rad}(abc)$ be the product of distinct prime numbers dividing $abc$. Oesterl\'e posed the question whether the numbers
    $$
    L=L(a,b,c)=\frac{\log\max(|a|,|b|,|c|)}{\log\text{rad}(abc)}
    $$
are bounded. This question was refined by Masser and Oesterl\'e who conjectured that for each $\varepsilon>0$ there exists a positive constant $K(\varepsilon)$ such that
    $$
    \max(|a|,|b|,|c|)\leq K(\varepsilon)\text{rad}(abc)^{1+\varepsilon}.
    $$
This is the well-known $abc$-conjecture, see \cite{Oesterle1988}. Lang in \cite[Page 196]{Lang} said that
\begin{quote}
\emph{
One of the most fruitful analogies in mathematics is that between the
integers $\Z$ and the ring of polynomials $F[t]$ over a field $F$.}
\end{quote}
The Stothers--Mason theorem is the polynomial analogue of $abc$-conjecture.
Let $P$ be a polynomial. The radical ${\rm rad}(P)$ is the product of distinct linear factors of $P$.
Let $a$, $b$ and $c$ be relatively prime polynomials such that not all of them are identically zero. The Stothers--Mason theorem~\cite{mason1984},~\cite{stothers1981}, see also e.g.,~\cite{Snyder} states that if they satisfy $a+b=c$, then
    $$
    \max\{\deg(a),\deg(b),\deg(c)\}\leq \deg({\rm rad}(abc))-1.
    $$

A difference analogue of the Stothers-Mason theorem or difference $abc$ theorem for polynomials was given by Ishizaki et al. in \cite[Theorem~3.1]{IKLT}. Let $a,b$ and $c$ be relatively prime polynomials in $\C[z]$ such that $a+b=c$ and such that $a,b$ and $c$ are not all constants. Then,
    $$
    \max\{\deg(a),\deg(b),\deg(c)\}\leq \tilde{n}_\kappa(a)+\tilde{n}_\kappa(b)+\tilde{n}_\kappa(c)-1,
    $$
where $\kappa\in\C\setminus\{0\}$, and
    $$
    \tilde{n}_\kappa(p)=\sum_{w\in\C}(\text{ord}_w(p)-\min\{\text{ord}_w(p),\text{ord}_{w+\kappa}(p)\})
    $$
for a polynomial $p$ by $\text{ord}_w(p)$ denoting the order of zeros of $p$ at $z=w$.

Another difference analogue of the Stothers-Mason theorem was given by Ishizaki and Wen in \cite[Theorem~3.5]{IW2022}. Let $a, b$ and $c$ be relatively shifting prime polynomials, see \cite[pages~736-737]{IW2022}, such that $a+b=c$ and such that $a, b$ and $c$ are not all constants. Then
    \begin{equation}\label{diffabc.eq}
    \max\{\deg(a),\deg(b),\deg(c)\}\leq \deg(\text{rad}_\Delta(abc))-1,
    \end{equation}
where $\text{rad}_\Delta(abc)$ is difference radical of $abc$ defined in \cite[Section~3]{IW2022}.

The Stothers-Mason theorem shows that the maximum degree of $a,b,c$ can be controlled by the number of distinct zeros of $abc$. The difference analogue of the Stothers-Mason theorem
reveals that the maximum degree of $a,b,c$ can be controlled by the number of distinct lines of zeros of $abc$ (we set two points $z_1,z_2$ are lying in one line if $z_1=z_2\pm1$, but two points $z_3,z_4$ are in two different lines even if they are $z_3=z_4$).
A meromorphic function's Nevanlinna characteristic function is dominated by the distribution of distinct zeros of $f-a_j$ for $a_j\in\C$ according to Nevanlinna second main theorem. It is nature to consider whether a meromorphic function's Nevanlinna characteristic function can be dominated by the distribution of distinct lines of zeros of $f-a_j$ for $a_j\in\C$. In order to study this problem, we give new definitions of length of poles and length of zeros, which are difference analogues of multiplicity of poles and multiplicity of zeros,
see Section 3 and Section 4.

Our purpose of this paper is to generalize the inequality \eqref{diffabc.eq} for entire functions of order less than 1 and give a difference analogue of truncated version of Nevanlinna second main theorem by new definitions of length of poles and length of zeros. In Section 2, we introduce a new definition of length of poles of a given meromorphic function in a domain. We state the difference analogue of Stothers-Mason theorem for entire functions of order less than 1 in Section 3, which can be seen as difference ``abc" theorem for entire functions, while we also generalize the result for $m+1$ entire functions of order less than 1. In Section 4, we give the difference analogue of truncated version of Nevanlinna second main theorem. It reveals that subnormal meromorphic functions such that $\Delta f\neq 0$ cannot have too many points with long length. We consider difference Fermat functional equations in Section 5, which is an application of the difference analogue of Stothers-Mason theorem for entire functions of order less than 1.

\section{Length of poles}

For a function $f$, we denote by $\Delta f(z)=f(z+1)-f(z)$ the difference operator. Let $n$ be a nonnegative integer. Define $\Delta^n f(z)=\Delta(\Delta^{n-1} f(z))$ for $n\geq 1$, and write $\Delta^0 f=f$.
Define $z^{\underline{0}}=1$ and
    $$
    z^{\underline{n}}=z(z-1)\cdots(z-n+1)=n!\begin{pmatrix}
    z\\
    n
    \end{pmatrix},\quad n=1, 2, 3, \dots,
    $$
which is called a \emph{falling factorial}, see \cite[Page 25]{M-Thomson1933} or
\cite[Subsection 2.1]{Ishizaki-Wen}.
Consider the formal series of the form
    $$
    Y_1(z)=\sum_{n=0}^\infty \tilde{a}_n z^{\underline{n}},\quad \tilde{a}_n\in\mathbb C,\quad n=0,1,2, \dots,
    $$
which is called the \emph{binomial series} or the \emph{factorial series}.
More information on convergence of binomial series in connection with
the order of growth of entire function can be found in \cite{Ishizaki-Wen}.
Following the view of Boole~\cite[Page 7]{Boole}, we adopt the notation
    \begin{equation*}\label{-n.eq}
    z^{\underline{-n}}=\frac{1}{z(z+1)\cdots(z+n-1)},\quad~n=1,2,\ldots.
    \end{equation*}
It yields that
    \begin{equation*}\label{diff-minus.eq}
    \Delta z^{\underline{-n}}=(z+1)^{\underline{-n}}-z^{\underline{-n}}=-nz^{\underline{-(n+1)}},
    \end{equation*}
which is an analogue of derivative of inverse power series $(z^{-n})'=-nz^{-(n+1)}$, see e.g.,~\cite[Page 5]{Norlund1924}.
We add that $z^{\overline{n}}=z(z+1)\cdots(z+n-1)$ is called a \emph{raising factorial}, which coincides with the Pochhammer symbol~\cite[Page 171]{Pochhammer1893} denoted by $(z)_{n}$ in this paper, see e.g.,~\cite[Page 143]{Kohno1999},~\cite{Weniger2010}, and note that $z^{\underline{-n}}$ is characterized by $z^{\underline{-n}}=1/z^{\overline{n}}=1/(z)_{n}=\Gamma(z)/\Gamma(z+n)$

Let $f$ be meromorphic in a domain $G\subset\C$ and let $n\in\N$.
Suppose that $z_0, z_0-1,\ldots, z_0-n\in G$.
The point $z_0$ is a \emph{pole of $f(z)$ with length $n$ in $G$} provided $z_0-k$ are poles of $f(z)$ for every $k\in\N$ such that $0\leq k<n$, but $z_0-n$ is not the pole of $f(z)$ .
In particular, if $z_0$ is a pole of $f$ with length 1, that is, $z=z_0$ is the pole of $f$ but $z=z_0+1$ is not, then we say $z_0$ is a pole of $f$ with \emph{simple length}.
Note here that length of poles is a difference analogue
of multiplicity of poles.

\begin{example}
The function $f(z)=2^z=e^{z\ln 2}$ has no pole, whose length of every pole is 0. Let $g$ be defined as $g(z)=1/(z^2(z+1)(z+2))2^z$, which can be written
    $$
   g(z)=z^{\underline{-3}}\cdot z^{\underline{-1}}2^z=(z+1)^{\underline{-2}}\cdot (z^{\underline{-1}})^22^z.
    $$
Then $z=0$ is a pole of $g$ with length $3$ in $\C$. In addition, $z=-1$ and $z=-2$ are the poles of $g$ with length 2 and 1 in $\C$, respectively.
\end{example}

The length of a given pole of a meromorphic function in a domain $G$ can be infinite. For example, we consider the Euler Gamma function $\Gamma(z)$. It is known that $\Gamma(z)$ is a transcendental meromorphic function that has simple poles at $0, -1, -2, \dots$. It gives that for any $n\in\N$, $z=-n$ is a pole of $\Gamma(z)$ with infinite length in $\C$, see e.g.,~\cite[Page 236]{WW1927}.

\begin{proposition}
If $z_0$ is a pole of a meromorphic function with infinite length in $\C$, then the order of growth of $f$ is $\rho(f)\geq 1$.
\end{proposition}

\begin{proof}
We define counting function $n(r,f)$ denoting the number of poles of $f$ in the disk $|z|\leq r$. Since $z_0$ is a pole of $f$ with infinite length, there exists an $R>|z_0|$ such that $n(r,f)\geq r+O(1)$ for $r>R$. Then
    $$
    N(r,f)=\int_0^r \frac{n(t,f)}{t}\,dt\geq \int_{r/2}^{r}\frac{n(t,f)}{t}\,dt\geq n(r/2,f)\log 2
    $$
for $r\geq 2R$. Therefore,
     $$\rho(f)=\varlimsup_{r\to \infty} \frac{\log T(r,f)}{\log r}\geq \varlimsup_{r\to \infty} \frac{\log N(r,f)}{\log r}\geq 1.
     $$
\end{proof}

It is well known that $z=z_0$ is a pole of a meromorphic function $f$ with multiplicity $n$ if and only if $f$ is of
the form $f=(z-z_0)^{-n}g(z)$, where $g(z)$ is analytic at $z=z_0$ and $g(z_0)\neq 0$. The difference analogue of this result is given as follows.

\begin{theorem}\label{pole.theorem}
Let $f$ be meromorphic in the domain $G\subset\C$ and let $n\in\N$. Suppose that $z_0, z_0-1,\ldots, z_0-n\in G$.
The point $z_0$ is a pole of $f(z)$ with length $n$ in $G$ if and only if there exists a meromorphic function $g(z)$ in $G$ such that
    $$
    f(z)=(z-z_0)^{\underline{-n}}g(z),
    $$
where $g(z)$ is analytic at $z=z_0-n$ and $g(z_0-j)\neq 0$ for $j=0,1,\ldots,n-1$.
\end{theorem}

\begin{proof}
From the definition of length of the pole, we get our assertion.
\end{proof}

Let $f$ be meromorphic in the domain $G\subset\C$ and let $m\in\N$. Suppose that $z_0, z_0-1,\ldots, z_0-m\in G$.
We note that if $z=z_0$ is a pole of $f$ with length $m$ in $G$, $z=z_0$ is not necessarily a pole of $\Delta f$ with length $m+1$. In order to illustrate it, we give
the following example.

\begin{example}
 Let $f(z)=z^{\underline{-2}}(2z+1)$. It is clear that $f$ has a pole at $z=0$ with length $2$, but $z=0$ is a pole of $\Delta f(z)=-2/(z^2+2z)$ with simple length.
 Let $g(z)=1/(z^2+2z)$. It shows that $g$ has a pole
 at $z=0$ with simple length. However, $z=0$ is a pole
 of $\Delta g(z)=-z^{\underline{-4}}(2z+3)$ with length $4$.
\end{example}

 \section{Difference Analogue of the Stothers-Mason theorem for entire functions}

Let $P$ be a polynomial. The radical ${\rm rad}(P)$ is the product of distinct linear factors of $P$.
Let $a$, $b$ and $c$ be relatively prime polynomials such that not all of them are identically zero. The Stothers--Mason theorem~\cite{mason1984},~\cite{stothers1981}, see also e.g.,~\cite{Snyder} states that if they satisfy $a+b=c$, then
    \begin{equation*}\label{smt}
    \max\{\deg(a),\deg(b),\deg(c)\}\leq \deg({\rm rad}(abc))-1.
    \end{equation*}

In order to state a difference analogue of the Stothers-Mason theorem,
let us recall the definition of length of zeros, which is the difference analogue of multiplicity of zeros, see \cite[Section~2]{IW2022}.
We change the title with the same meaning.
Let $f$ be analytic in a domain $G\subset\C$ and let $n\in\N$.
Suppose that $z_0, z_0+1,\ldots, z_0+n\in G$.
The point $z_0$ is called a \emph{zero of $f(z)$ with length $n$ in $G$} provided $f(z_0)$ and all difference $\Delta^k f(z_0)$ vanish for every $0\leq k<n$, but $\Delta^n f(z_0)\neq 0$. In general, the point $z_0$ is called a \emph{$a$-point of $f(z)$ with length $n$ in $G$} provided $f(z_0)-a$ and all difference $\Delta^k f(z_0)$ vanish for every $0\leq k<n$, but $\Delta^n f(z_0)\neq 0$.
In particular, if $z_0$ is an $a$-point of $f$ with height 1, then we also call $z_0$ is an $a$-point of $f$ with \emph{simple length}.

Let $P$ be a polynomial with degree $p$, and $z_1$ be a zero of $P$ with length $n_1$ and $P(z_1-1)\neq 0$. Theorem~\cite[Theorem 2.4]{IW2022} states that there exists a polynomial $P_1$ with degree $p-n_1$ such that $P(z)=(z-z_1)^{\underline{n_1}}P_1(z)$, where $P_1(z_1+n_1)\neq 0$.
Let $z_2$ be a zero of $P_1$ with length $n_2$ and $P(z_2-1)\neq 0$.
Then there exists a polynomial $P_2$ with degree $p-n_1-n_2$ such that $P_1(z)=(z-z_2)^{\underline{n_2}}P_2(z)$, where $P_2(z_2+n_2)\neq 0$.
Repeating this argument for finitely many times, we see that $P(z)$ can be written uniquely as
\begin{equation}
P(z)=A\prod_{j=1}^N(z-z_j)^{\underline{n_j}},\label{5.001}
\end{equation}
where $A$ is a nonzero constant, and $p=n_1+\cdots+n_N$.
Note that it is possible that $z_j=z_k$ even though $j\ne k$ in \eqref{5.001}.
We define the \emph{difference radical ${\rm rad}_{\Delta}(P)$} by product of these linear factors, i.e.,
\begin{equation}
{\rm rad}_{\Delta}(P)=\prod_{j=1}^N(z-z_j).\label{5.002}
\end{equation}

We denote the closed disc of radius $r>0$ centred at $z_0\in\C$ by
$\overline{D}(z_0,r):=\{z\in\C: |z-z_0|\leq r\}$. Suppose that $f$ is a meromorphic function in $\C$, and $z=z_1\in \overline{D}(0,r)$ is a $a$-point of $f$ with length $n_1$, while $z=z_1-1$ is not an $a$-point of $f$ in $\overline{D}(0,r)$ or $z_1-1\not\in\overline{D}(0,r)$. Then by Theorem~\cite[Theorem 2.4]{IW2022},
there exists a meromorphic function $f_1$ such that
    $$
    f(z)-a=(z-z_1)^{\underline{n_1}}f_1(z),
    $$
where $f_1(z_1+n_1)\neq 0$ or $z_1+n_1\not\in \overline{D}(0,r)$, and $f_1(z)$ is analytic at $z=z_1+j$ for $j=0,1,\ldots,n_1-1$. Suppose that $z=z_2\in \overline{D}(0,r)$ is a zero of $f_1$ with length $n_2$ and $f_1(z_2-1)\neq 0$ in $\overline{D}(0,r)$ or $z_2-1\not\in\overline{D}(0,r)$. Then there exists a meromorphic function $f_2$ such that
    $$
    f_1(z)=(z-z_2)^{\underline{n_2}}f_2(z),
    $$
where $f_2(z_2+n_2)\neq 0$ or $z_2+n_2\not\in\overline{D}(0,r)$, and $f_2(z)$ is analytic at $z=z_2+j$ for $j=0,1,\ldots,n_2-1$. Repeating this argument for countable many times, we have there exists a meromorphic function $F$ such that
    \begin{equation}\label{fF.eq}
    f(z)-a=\prod_{z_j\in\overline{D}(0,r)}(z-z_j)^{\underline{n_j}}F(z),
    \end{equation}
where $F(z)$ is zero free in $\overline{D}(0,r)$ and $F(z)$ is analytic at each $z=z_j$ in \eqref{fF.eq}. Note that it is possible that $z_j=z_k$ even though $j\ne k$ in \eqref{fF.eq}.
We denote the number of such points $z_j$ in the disc $\overline{D}(0,r)$ in \eqref{fF.eq} by $\overline{n}_\Delta(r,1/(f-a))$. We call such points $z_j$ are \emph{initial shifting $a$-points} of $f$ in $\overline{D}(0,r)$. For example, let
    $$
    f(z)=z^2(z-1)^3(z-2)^4e^z=z^{\underline{3}}z^{\underline{3}}(z-1)^{\underline{2}}(z-2)e^z,
    $$
we see that $\overline{n}_\Delta(r,1/f)=4$ when $r\geq 4$, and the initial shifting zeros of $f$ are $0,0,1,2$ in $\overline{D}(0,4)$. If $g(z)=\sin\pi z$, then $\overline{n}_\Delta(r,1/g)=1$ for $r\geq 1$, and the initial shifting zero of $g$ is $-[r]$, where $[r]$ denotes the integer part of $r$. It is obvious that if $f$ is not a periodic function with period one, then for $a_1,\ldots,a_q\in\C$, we have
    \begin{equation}\label{aD.eq}
    \sum_{j=1}^qn\left(r,\frac{1}{f-a_j}\right)\leq n\left(r,\frac{1}{\Delta f}\right)+\sum_{j=1}^q\overline{n}_{\Delta}\left(r,\frac{1}{f-a_j}\right)
    \end{equation}
by \cite[Corollary~2.6]{IW2022}.
In particular, if $f$ is a polynomial, then
    $$
    \overline{n}_{\Delta}\left(r,\frac{1}{f}\right)=\deg {\text{rad}_\Delta(f)}
    $$
for large $r$, where $\text{rad}_\Delta(f)$ is the difference radical of the polynomial $f$ defined in \eqref{5.002}, see \cite[Section~3]{IW2022}.
The corresponding
integrated counting function is defined in the usual way as
    $$
    \overline{N}_\Delta\left(r,\frac{1}{f-a}\right):=\int_0^r\frac{\overline{n}_\Delta\left(t,\frac{1}{f-a}\right)
    -\overline{n}_\Delta\left(0,\frac{1}{f-a}\right)}{t}\,dt
    +\overline{n}_\Delta\left(0,\frac{1}{f-a}\right)\log r.
    $$

\begin{remark}
We note here that our definition on $\overline{n}_\Delta(r,1/f)$ is not equivalent to $\tilde{n}^{[q]}_\kappa(r,1/f)$ in \cite[(5.5)]{IKLT} even when $q=1$ and $\kappa=1$ defined by
    $$
    \tilde{n}^{[1]}_1\left(r,\frac{1}{f}\right)=\sum_{w\in\overline{D}(0,r)}\left(
    {\rm ord}_w(f)-\min\{{\rm ord}_w(f),{\rm ord}_{w+1}(f)\}\right),
    $$
where ${\rm ord}_w(f)$ is the multiplicity of zeros of $f$ at $z=w$. For example, let $f=\sin\pi z$. It is known that all zeros of $f$ are integers. While $\overline{n}_\Delta(r,1/f)=1$ and $\tilde{n}^{[1]}_1(r,1/f)=0$ for any $r>0$. In fact, according to \cite[Theorem~2.4]{IW2022} and idea in \cite{IKLT}
    \begin{equation}\label{Szero.eq}
    \begin{split}
    \overline{n}_\Delta\left(r,\frac{1}{f}\right)&=\sum_{w\in\overline{D}(0,r)}\left(
    {\rm ord}_w(f)-\min\{{\rm ord}_w(f),{\rm Ord}_{w-1\in\overline{D}(0,r)}(f)\}\right)\\
    &=\sum_{\substack{|w|\leq r\\ |w-1|\leq r}}\left(
    {\rm ord}_w(f)-\min\{{\rm ord}_w(f),{\rm ord}_{w-1}(f)\}\right)+
    \sum_{\substack{|w|\leq r\\ |w-1|> r}}{\rm ord}_w(f)
    \end{split}
    \end{equation}
where ${\rm Ord}_{w-1\in\overline{D}(0,r)}(f)$ denotes the multiplicity of zeros of $f$ at $z=w-1$ when $w-1$ are in the closed disc $\overline{D}(0,r)$, otherwise ${\rm Ord}_{w-1\in\overline{D}(0,r)}(f)=0$. Obviously, we have the equality \eqref{Szero.eq} is equivalent to
    \begin{equation*}
    \begin{split}
    \overline{n}_\Delta\left(r,\frac{1}{f}\right)&=\sum_{w\in\overline{D}(0,r)}\left(
    {\rm ord}_w(f)-\min\{{\rm ord}_w(f),{\rm Ord}_{w+1\in\overline{D}(0,r)}(f)\}\right)\\
    &=\sum_{\substack{|w|\leq r\\ |w+1|\leq r}}\left(
    {\rm ord}_w(f)-\min\{{\rm ord}_w(f),{\rm ord}_{w+1}(f)\}\right)+
    \sum_{\substack{|w|\leq r\\ |w+1|> r}}{\rm ord}_w(f).
    \end{split}
    \end{equation*}
\end{remark}

Similarly, we define $\overline{n}_\Delta(r,f)$ as follows. Suppose that $f$ is a meromorphic function in $\C$, and $z=z_1\in \overline{D}(0,r)$ is a pole of $f$ with length $n_1$, while $z=z_1+1$ is not a pole of $f$ in $\overline{D}(0,r)$ or $z_1+1\not\in\overline{D}(0,r)$. Then by Theorem~\ref{pole.theorem}
there exists a meromorphic function $f_1$ such that
    $$
    f(z)=(z-z_1)^{\underline{-n_1}}f_1(z),
    $$
where $z=z_1-n_1$ is not the pole of $f_1$ in $\overline{D}(0,r)$ or $z_1-n_1\not\in\overline{D}(0,r)$, and $f_1(z)$ does not vanish at $z=z_1-j$ for $j=0,1,\ldots,n_1-1$.
Suppose that $z=z_2\in \overline{D}(0,r)$ is a pole of $f_1$ with length $n_2$ and $z_2+1$
is not a pole of $f_1(z)$ in $\overline{D}(0,r)$ or $z_2+1\not\in\overline{D}(0,r)$. Then there exists a meromorphic function $f_2$ such that
    $$
    f_1(z)=(z-z_2)^{\underline{-n_2}}f_2(z),
    $$
where $z=z_2-n_2$ is not the pole of $f_2$ in $\overline{D}(0,r)$ or $z_2-n_2\not\in\overline{D}(0,r)$, and $f_2(z)$ does not vanish at $z=z_2-j$ for $j=0,1,\ldots,n_2-1$. Repeating this argument for countable many times, we have there exists a meromorphic function $F$ such that
    \begin{equation}\label{fFpole.eq}
    f(z)=\prod_{z_j\in\overline{D}(0,r)}(z-z_j)^{\underline{-n_j}}F(z),
    \end{equation}
where $F$ is analytic in $\overline{D}(0,r)$ and $F$ does not vanish at $z=z_j$ in \eqref{fFpole.eq}. Note that it is possible that $z_j=z_k$ even though $j\ne k$ in \eqref{fFpole.eq}. We denote the number of such points $z_j$ in the disc $\overline{D}(0,r)$ in \eqref{fFpole.eq} by $\overline{n}_\Delta(r,f)$. We call such points $z_j$ \emph{initial shifting poles} of $f$ in $\overline{D}(0,r)$. For example, let
    $$
    f(z)=\frac{e^z}{z^2(z+1)^4(z+2)^3}=z^{\underline{-3}}z^{\underline{-3}}(z+1)^{\underline{-2}}(z+1)^{\underline{-1}}e^z.
    $$
We see that $\overline{n}_\Delta(r,f)=4$ for $r>3$. The initial shifting poles of $f$ are $z=0,0,-1,-1$ in $\overline{D}(0,r)$ for $r>3$. If $g=\Gamma(z)$, then $\overline{n}_\Delta(r,g)=1$ for $r>0$. It is clear that the initial shifting pole of $g$ is $z=0$ in $\overline{D}(0,r)$ for any $r>0$. Obviously, from the idea in \cite{IKLT}, we obtain that
    \begin{equation}\label{Spole.eq}
    \begin{split}
    \overline{n}_\Delta\left(r,f\right)&=\sum_{w\in\overline{D}(0,r)}\left(
    {\rm ord}^-_w(f)-\min\{{\rm ord}^-_w(f),{\rm Ord}^-_{w+1\in\overline{D}(0,r)}(f)\}\right)\\
    &=\sum_{\substack{|w|\leq r\\ |w+1|\leq r}}\left(
    {\rm ord}^-_w(f)-\min\{{\rm ord}^-_w(f),{\rm ord}^-_{w+1}(f)\}\right)+
    \sum_{\substack{|w|\leq r\\ |w+1|>r}}{\rm ord}^-_w(f),
    \end{split}
    \end{equation}
where ${\rm ord}^-_{w}(0,r)(f)$ denotes the multiplicity of poles of $f$ at $z=w$, and ${\rm Ord}^-_{w+1\in\overline{D}(0,r)}(f)$ denotes the multiplicity of poles of $f$ at $z=w+1$ when $w+1\in\overline{D}(0,r)$, if $w+1\not\in\overline{D}(0,r)$ then ${\rm Ord}^-_{w+1\in\overline{D}(0,r)}(f)=0$. The equality \eqref{Spole.eq} is equivalent to
    \begin{equation*}
    \begin{split}
    \overline{n}_\Delta\left(r,f\right)&=\sum_{w\in\overline{D}(0,r)}\left(
    {\rm ord}^-_w(f)-\min\{{\rm ord}^-_w(f),{\rm Ord}^-_{w-1\in\overline{D}(0,r)}(f)\}\right)\\
    &=\sum_{\substack{|w|\leq r\\ |w-1|\leq r}}\left(
    {\rm ord}^-_w(f)-\min\{{\rm ord}^-_w(f),{\rm ord}^-_{w-1}(f)\}\right)+
    \sum_{\substack{|w|\leq r\\ |w-1|>r}}{\rm ord}^-_w(f).
    \end{split}
    \end{equation*}
In addition, the corresponding integrated counting function is defined as
    $$
    \overline{N}_\Delta\left(r,f\right):=\int_0^r\frac{\overline{n}_\Delta\left(t,f\right)
    -\overline{n}_\Delta\left(0,f\right)}{t}\,dt
    +\overline{n}_\Delta\left(0,f\right)\log r.
    $$

\begin{lemma}\label{common.lemma}
Suppose that $f$ is a meromorphic function of order $\rho(f)<1$. Then
    $$
    N\left(r,\frac{\Delta f}{f}\right)=\overline{N}_\Delta(r,f)+\overline{N}_\Delta\left(r,\frac{1}{f}\right)+O(1).
    $$
\end{lemma}

\begin{proof}
Since $\Delta f/f=f(z+1)/f(z)-1$, from \eqref{Szero.eq} and \eqref{Spole.eq} we have
    \begin{align*}
    n\left(r,\frac{\Delta f}{f}\right)=&\sum_{w\in\overline{D}(0,r)}\left({\rm ord}_w(f)-\min\{ {\rm ord}_w(f),{\rm ord}_{w+1}(f)\}\right)\\
    &+\sum_{w\in\overline{D}(0,r)}\left({\rm ord}^-_{w+1}(f)-\min\{ {\rm ord}^-_w(f),{\rm ord}^-_{w+1}(f)\}\right)\\
    =&\overline{n}_\Delta\left(r,\frac{1}{f}\right)-\sum_{\substack{|w|\leq r\\|w+1|>r}}\min\{ {\rm ord}_w(f),{\rm ord}_{w+1}(f)\}\\
    &+\overline{n}_\Delta(r,f)+\sum_{\substack{|w|\leq r\\|w+1|>r}}\left({\rm ord}^-_{w+1}(f)-
    \min\{ {\rm ord}^-_w(f),{\rm ord}^-_{w+1}(f)\}\right)
    -\sum_{\substack{|w|\leq r\\ |w-1|>r}}{\rm ord}^-_w(f).
    \end{align*}
By Riemann-Stieltjes integral, it follows that
    \begin{equation*}
    \begin{split}
    &N\left(r,\frac{\Delta f}{f}\right)-\overline{N}_\Delta(r,f)-\overline{N}_\Delta\left(r,\frac{1}{f}\right)\\
    =&O\left(\sum_{\substack{|a_j|\leq r\\ |a_j+1|>r}}\log\frac{r}{|a_j|}
    -\sum_{\substack{|a_j|\leq r\\ |a_j+1|>r}}\log\frac{r}{|a_j+1|}
    -\sum_{\substack{|b_j|\leq r\\ |b_j+1|>r}}\log\frac{r}{|b_j+1|}
    +\sum_{\substack{|b_j|\leq r\\ |b_j-1|>r}}\log\frac{r}{|b_j|}\right),
    \end{split}
    \end{equation*}
where $\{a_j\}$ are the sequence of zeros of $f$ and $\{b_j\}$ are the sequence of poles of $f$, with due to counting multiplicity. In addition, by using the inequality $\log(1+x)\leq x$ and $\rho(f)<1$,
    $$
    \sum_{\substack{|a_j|\leq r\\ |a_j+1|>r}}\log\frac{r}{|a_j|}
    =\sum_{\substack{|a_j|\leq r\\ |a_j+1|>r}}\log\left(\frac{r-|a_j|}{|a_j|}+1\right)
    \leq\sum_{\substack{|a_j|\leq r\\ |a_j+1|>r}}\frac{r-|a_j|}{|a_j|}
    \leq \sum_{\substack{|a_j|\leq r\\ |a_j+1|>r}}\frac{1}{|a_j|}<\infty.
    $$
Similarly, we have
    $$
    \sum_{\substack{|a_j|\leq r\\ |a_j+1|>r}}\log\frac{|a_j+1|}{r}
    \leq\sum_{\substack{|a_j|\leq r\\ |a_j+1|>r}}\frac{|a_j+1|-r}{r}
    \leq \sum_{\substack{|a_j|\leq r\\ |a_j+1|>r}}\frac{1}{|a_j|}<\infty,
    $$
    $$
    \sum_{\substack{|b_j|\leq r\\ |b_j+1|>r}}\log\frac{|b_j+1|}{r}
    \leq \sum_{\substack{|b_j|\leq r\\ |b_j+1|>r}}\frac{|b_j+1|-r}{r}
    \leq \sum_{\substack{|b_j|\leq r\\ |b_j+1|>r}}\frac{1}{|b_j|}
    <\infty
    $$
and
    $$
    \sum_{\substack{|b_j|\leq r\\ |b_j-1|>r}}\log\frac{r}{|b_j|}\leq
     \sum_{\substack{|b_j|\leq r\\ |b_j-1|>r}}\log\left(1+\frac{1}{|b_j|}\right)
     \leq  \sum_{\substack{|b_j|\leq r\\ |b_j-1|>r}}\frac{1}{|b_j|}<\infty
    $$
Therefore, we prove our assertion.
\end{proof}

Suppose that $f$ and $g$ are entire functions, and suppose that $z_1\in\C$ is a zero of $f$ with length $m$ and $z_2\in\C$ is a zero of $g$ with length $n$.
Then from Theorem \cite[Theorem~2.1]{IW2022}, there exist entire functions $F$ and $G$ such that
$f(z)=(z-z_1)^{\underline{m_1}}F(z)$ and $g(z)=(z-z_2)^{\underline{n_1}}G(z)$, where $1\leq m_1\leq m$ and $1\leq n_1\leq n$. If for some $m_1$, $n_1$,
    $$
    f(z)g(z)=(z-z_0)^{\underline{m_1+n_1}}F(z)G(z),
    $$
where $z_0$ is $z_1$ or $z_2$, then $z-z_0$ is called the \emph{common shifting divisor} of $f$ and $g$, which is the analogue of classical common divisor.
 If $f$ and $g$ do not have any nonconstant shifting common divisors, then $f$ and $g$ are called \emph{relatively shifting prime}. Obviously, if $f$ and $g$ are shifting prime, then
    $$
    \overline{N}_{\Delta}\left(r,\frac{1}{fg}\right)=\overline{N}_{\Delta}\left(r,\frac{1}{f}\right)+
    \overline{N}_{\Delta}\left(r,\frac{1}{g}\right).
    $$

 \begin{remark}\label{prime1.re}
  Suppose that $z=z_0$ is a zero of an entire function $f$ with simple length. If entire functions $f$ and $g$ do not have shifting common divisor $z-z_0$, then $z=z_0$ is a zero of $g$ with simple length or $z=z_0$ is not the zero of $g$.
 \end{remark}

  \begin{remark}\label{prime2.re}
   Suppose that $z=z_0$ is a zero of an entire function $f$ with length $m\geq 2$, and a zero of an entire function $g$ with length $n\geq 2$. Then $z-z_0$ is a shifting common divisor of $f$ and $g$.
 \end{remark}

\begin{remark}\label{commonzero.re}
 From Remarks \ref{prime1.re} and \ref{prime2.re}, we have if entire functions $f$ and $g$ are relatively shifting prime, then there exists $h$ such that $f/h$ and $g/h$
 are relatively prime entire functions, where all the zeros of $h$ are at most with simple length.
\end{remark}

Now it is possible for us to rewrite the difference analogue of the Stothers-Mason theorem \cite[Theorem~3.4]{IW2022} as follows.

\bigskip
\noindent
\textbf{Difference analogue of the Stothers-Mason theorem}.
\emph{Let $a,b$ and $c$ be relatively shifting prime polynomials such
that $a+b=c$ and such that $a,b$ and $c$ are not all constants. Then there exists $R>0$ such that for $r>R$
    \begin{equation}\label{abc.eq}
    \begin{split}
    \max\{\deg(a),\deg(b),\deg(c)\}&\leq \overline{n}_\Delta\left(r,\frac{1}{abc}\right)-1\\
    &=\overline{n}_\Delta\left(r,\frac{1}{a}\right)
    +\overline{n}_\Delta\left(r,\frac{1}{b}\right)
    +\overline{n}_\Delta\left(r,\frac{1}{c}\right)
    -1.
    \end{split}
    \end{equation}}

In the following, we proceed to generalize difference analogue of the Stothers-Mason theorem to entire functions of order $<1$. Let us recall some useful notation.

Suppose that $f$ is a nonconstant meromorphic function in the complex plane. Write $f=a/b$ as the quotient
of two entire functions $a$ and $b$ without common zeros.
It is well known that Cartan characteristic function of $f$ differs from its Nevanlinna characteristic function by a bounded term. The Cartan characteristic function, see \cite[Section 2.5.2]{Noguchi-Winkelmann}, is denoted by
    $$
    T_{f_1,f_2,\ldots,f_n}(r):=\frac{1}{2\pi}\int_0^{2\pi}\log u(re^{i\theta})\,d\theta-\log u(0),~u(re^{i\theta})=\max\{|f_1(re^{i\theta})|,\ldots,|f_n(re^{i\theta})|\},
    $$
where $f_1,\ldots,f_n$ are entire functions without common zeros. In order to study the relation on entire functions with common zeros in the section, we analogize the Cartan characteristic function and define
    $$
    \widetilde{T}_{a_1,a_2,\ldots,a_n}(r):=\frac{1}{2\pi}\int_0^{2\pi}\log u(re^{i\theta})\,d\theta-\log u(0),~u(re^{i\theta})=\max\{|a_1(re^{i\theta})|,\ldots,|a_n(re^{i\theta})|\},
    $$
where $a_1,\ldots,a_n$ are entire functions such that $u(0)\neq 0$. Suppose that $u(0)=0$, let the Laurent expansion of $a_j(z)$ at origin be
    $$
    a_j(z)=c_{\lambda_j}z^{\lambda_j}+c_{\lambda_j+1}z^{\lambda_j+1}+\cdots,\quad c_{\lambda_j}\neq 0,
    $$
where $\lambda_j=n(0,1/a_j)$ for $j=1,\ldots,n$. Hence, we define
    $$
    \widetilde{T}_{a_1,a_2,\ldots,a_n}(r):=\frac{1}{2\pi}\int_0^{2\pi}\log u(re^{i\theta})\,d\theta-\max_{1\leq j\leq n}{\log |c_{\lambda_j}|}.
    $$

\begin{remark}\label{we-de.remark}
According to the idea in \cite{GH2004}, we note that $\widetilde{T}_{a_1,a_2,\ldots,a_n}(r)$ is well-defined even if $a_1,\ldots,a_n$ has common zeros. Let us suppose that the great common division of $a_1,\ldots,a_n$ is $h$. Then $b_j=a_j/h$ have no common zeros for $j=1,\ldots,n$. Obviously,
    \begin{equation}\label{Th.eq}
    \widetilde{T}_{a_1,a_2,\ldots,a_n}(r)=T_{b_1,\ldots,b_n}(r)-\frac{1}{2\pi}\int_0^{2\pi}\log h(re^{i\theta})d\theta+O(1).
    \end{equation}
Now suppose that $h$ has a finite number of zeros on $|z|=r>0$.
Set
a curve $\gamma=\gamma(r,\delta)$ consisting of arcs of $|z|=r$ and
small indentations of radius $\delta$
about each zero of $h$ on $|z|=r$. In this case, \eqref{Th.eq} holds when the
path of integration is $\gamma$ instead of $|z| = r$. As $\delta\to 0$, on each small indentation
the integrand $\log h$ is $O(\log\delta)$, and the length of
the indentation is $O(\delta)$, and so the corresponding integrals around each indentation
tend to zero. Since the whole curve $\gamma$ approaches the circle $|z|=r$ as $\delta\to 0$, it shows us that \eqref{Th.eq} holds for all positive $r$, and $\widetilde{T}_{a_1,a_2,\ldots,a_n}(r)$ is well-defined.
\end{remark}

\begin{lemma}\label{T.lemma}
Suppose that $a$ and $b$ are entire functions, and $n_{a,b}(r)$ denotes the number of common zeros of $a$ and $b$ counting multiplicity in the closed disc $\overline{D}(0,r)$. Then
    $$
    \widetilde{T}_{a,b}(r)=T\left(r,\frac{a}{b}\right)+R_{a,b}(r)+O(1),
    $$
where $R_{a,b}(r)$ is the integrated counting function of $n_{a,b}(r)$ defined by
    $$
    R_{a,b}(r):=\int_0^r\frac{n_{a,b}(t)-n_{a,b}(0)}{t}\,dt+n_{a,b}(0)\log r.
    $$
\end{lemma}

\begin{proof}
By the Jensen formula and the first main theory of Nevanlinna, it yields that
    \begin{align*}
    T\left(r,\frac{a}{b}\right)&=m\left(r,\frac{a}{b}\right)+N\left(r,\frac{a}{b}\right)\\
    &=\frac{1}{2\pi}\int_0^{2\pi}\log^+\left|\frac{a(re^{i\theta})}{b(re^{i\theta})}\right|\,d\theta+
    N\left(r,\frac{1}{b}\right)-R_{a,b}(r)\\
    &=\frac{1}{2\pi}\int_0^{2\pi}\log^+\left|\frac{a(re^{i\theta})}{b(re^{i\theta})}\right|\,d\theta+
    \frac{1}{2\pi}\int_0^{2\pi}\log|b(re^{i\theta})|\,d\theta-R_{a,b}(r)+O(1)\\
    &=\widetilde{T}_{a,b}(r)-R_{a,b}(r)+O(1).
    \end{align*}
\end{proof}
\noindent
Now by using Lemma \ref{T.lemma}, it is easy to see that \eqref{abc.eq} is written as
    $$
    \widetilde{T}_{a,b,c}(r)\leq \overline{N}_\Delta\left(r,\frac{1}{abc}\right)-\log r+O(1)
    $$
for large $r$. This inequality inspires us to generalize  Stothers-Mason theorem for transcendental entire functions.
Now let us state our result as follows.

\begin{theorem}\label{functionABC.the}
Let $a,b$ and $c$ be relatively shifting prime entire functions of order less than 1 such that
    $$
    a+b=c
    $$
and such that $a,b$ and $c$ are not all constants. Then there exists $R>0$ such that for any $\varepsilon>0$
    \begin{equation*}
    \begin{split}
    \widetilde{T}_{a,b,c}(r)&\leq \overline{N}_\Delta\left(r,\frac{1}{abc}\right)
    -(1-\delta-\varepsilon)\log r\\
    &=\overline{N}_\Delta\left(r,\frac{1}{a}\right)
    +\overline{N}_\Delta\left(r,\frac{1}{b}\right)
    +\overline{N}_\Delta\left(r,\frac{1}{c}\right)
    -(1-\delta-\varepsilon)\log r
    \end{split}
    \end{equation*}
holds for $r>R$ outside an exceptional set on $r$ with finite logarithmic measure, where $\delta=\max\{\rho(a),\rho(b),\rho(c)\}$.
\end{theorem}

\begin{proof}
We denote a set of common zeros of $a$ and $b$ by $\{S_n\}$.
Since the common zeros of $a$ and $b$ are zeros of $c$, all the elements in $\{S_n\}$ are zeros of $c$.
We divide the set of zeros counting multiplicity of $a$ (or $b$, $c$) into two subsets, which are denoted by $\{A_n\}$ (or $\{B_n\}$, $\{C_n\}$) and $\{S_n\}$, respectively.
For example, $z_0$ is a zero of $a$ with multiplicity $3$ and a zero of $b$ with multiplicity $2$, $z_0$ belongs to both sets $\{S_n\}$ and $\{A_n\}$, while $z_0\not\in \{B_n\}$.
Note here that the set $\{S_n\}$ may be empty. Obviously, we see that there is no common elements in each two sets of $\{A_n\},\{B_n\},\{C_n\}$. From Remark~\ref{prime1.re}, we know that if $\{S_n\}$ is not empty, all the points in the set $\{S_n\}$ are zeros of $a$ or $b$, $c$ with simple length, and it is impossible that there exists an element $z\in\{A_n\}\cup\{B_n\}\cup\{C_n\}$ such that $z\pm1\in\{S_n\}$.

For any $z'\in\{A_n\}\cup\{B_n\}\cup\{C_n-1\}$, it is known that $z'$ is a zero of $a(z)b(z)c(z+1)$
denoted by $ab\overline{c}$. We construct a canonical product $h_1$ whose zeros are from $z\in\{A_n\}\cup\{B_n\}\cup\{C_n-1\}$ with multiplicity ${\rm ord}_z(ab\overline{c})-\min\{{\rm ord}_z(ab\overline{c}),{\rm ord}_{z-1}(ab\overline{c})\}$. That is,
    $$
    h_1(z)=\prod_{w\in\{A_n\}\cup\{B_n\}\cup\{C_n-1\}}\left(1-\frac{z}{w}\right)^{{\rm ord}_w(ab\overline{c})-\min\{{\rm ord}_w(ab\overline{c}),{\rm ord}_{w-1}(ab\overline{c})\}}.
    $$
In addition, we construct $h_2$ as
    $$
    h_2=\prod_{w\in\{S_n\}}\left(1-\frac{z}{w}\right).
    $$
Now let us define an entire function $h=h_1h_2$. It is clear that $\rho(h)<1$ and
    \begin{equation}\label{N.eq}
    N\left(r,\frac{1}{h}\right)\leq\overline{N}_\Delta\left(r,\frac{1}{ab\overline{c}}\right)
    \leq\overline{N}_\Delta\left(r,\frac{1}{a}\right)+
    \overline{N}_\Delta\left(r,\frac{1}{b}\right)
    +\overline{N}_\Delta\left(r,\frac{1}{\overline{c}}\right).
    \end{equation}
Since
    \begin{align*}
    \overline{n}_\Delta\left(r,\frac{1}{\overline{c}}\right)
    -\overline{n}_\Delta\left(r,\frac{1}{c}\right)
    &=\sum_{\substack{|w|\leq r\\|w+1|>r\\|w-1|\leq r}}
    \left({\rm ord}_{w}(c)-\min\{{\rm ord}_w(c),{\rm ord}_{w+1}(c)\}\right)\\
    &-\sum_{\substack{|w|\leq r\\|w+1|\leq r\\|w-1|> r}}
    \left({\rm ord}_{w}(c)-\min\{{\rm ord}_w(c),{\rm ord}_{w+1}(c)\}\right)\\
    &+\sum_{\substack{|w|\leq r\\|w+1|>r}}
    \left({\rm ord}_{w+1}(c)-{\rm ord}_w(c)\right),
    \end{align*}
it follows by Riemann-Stieltjes integral that
    $$
    \overline{N}_\Delta\left(r,\frac{1}{\overline{c}}\right)=
    \overline{N}_\Delta\left(r,\frac{1}{c}\right)+O\left(\sum_{\substack{|\alpha_j|\leq r\\|\alpha_j+1|>r}}
    \log\frac{|\alpha_j+1|}{r}+\sum_{\substack{|\alpha_j|\leq r\\|\alpha_j\pm 1|>r}}
    \log\frac{r}{|\alpha_j|}\right),
    $$
where $\{\alpha_j\}$ are zeros of $c$ counting multiplicity. It gives us that
    $$
    \sum_{\substack{|\alpha_j|\leq r\\|\alpha_j+1|>r}}
    \log\frac{|\alpha_j+1|}{r}\leq\sum_{\substack{|\alpha_j|\leq r\\|\alpha_j+1|>r}}
    \log\left(1+\frac{|\alpha_j|+1-r}{r}\right)
    \leq\sum_{\substack{|\alpha_j|\leq r\\|\alpha_j+1|>r}}
    \frac{1}{|\alpha_j|}<\infty
    $$
and
    $$
    \sum_{\substack{|\alpha_j|\leq r\\|\alpha_j\pm 1|>r}}
    \log\frac{r}{|\alpha_j|}\leq\sum_{\substack{|\alpha_j|\leq r\\|\alpha_j\pm 1|>r}}
    \log\left(1+\frac{1}{|\alpha_j|}\right)\leq\sum_{\substack{|\alpha_j|\leq r\\|\alpha_j\pm 1|>r}}
    \frac{1}{|\alpha_j|}<\infty.
    $$
Therefore, it implies that
    \begin{equation}\label{overlinec.eq}
    \overline{N}_\Delta\left(r,\frac{1}{\overline{c}}\right)
    =\overline{N}_\Delta\left(r,\frac{1}{c}\right)+O(1)
    \end{equation}
for entire function $c$ of order $\rho(c)<1$.

Let $f=a/c$ and $g=b/c$. It is obvious that $f+g=1$ and $\Delta f+\Delta g=0$.
Hence
    $$
    \frac{a}{b}=\frac{f}{g}=-\frac{\Delta g/g}{\Delta f/f}
    =-\frac{h\Delta g/g}{h\Delta f/f},
    $$
Then $(\Delta g/g)h$ and $(\Delta f/f)h$ are entire, and
    $$
    k=-\frac{\Delta gh}{ga}=\frac{\Delta fh}{fb}
    $$
is also entire. Therefore, by Jensen formula, \cite[Theorem~5.1]{CF2009}, \eqref{N.eq}-\eqref{overlinec.eq} and $a,b,c$ are relatively shifting prime, it follows that
    \begin{equation*}
    \begin{split}
    \widetilde{T}_{a,b}(r)&=\frac{1}{2\pi}\int_0^{2\pi}\log\max\{|a(re^{i\theta})|,|b(re^{i\theta})|\}\,d\theta+O(1)\\
    &=\frac{1}{2\pi}\int_0^{2\pi}\log\max\{|ka|,|kb|\}\,d\theta-\frac{1}{2\pi}\int_0^{2\pi}\log|k|\,d\theta+O(1)\\
    &=\frac{1}{2\pi}\int_0^{2\pi}\log\max\bigg\{\left|\frac{\Delta f}{f}\right|,\left|\frac{\Delta g}{g}\right|\bigg\}\,d\theta
    +\frac{1}{2\pi}\int_0^{2\pi}\log|h|\,d\theta-\frac{1}{2\pi}\int_0^{2\pi}\log|k|\,d\theta+O(1)\\
    &\leq N\left(r,\frac{1}{h}\right)-(1-\delta-\varepsilon)\log r-N\left(r,\frac{1}{k}\right)\\
    &\leq\overline{N}_\Delta\left(r,\frac{1}{a}\right)
    +\overline{N}_\Delta\left(r,\frac{1}{b}\right)
    +\overline{N}_\Delta\left(r,\frac{1}{c}\right)
    -(1-\delta-\varepsilon)\log r\\
    &= \overline{N}_\Delta\left(r,\frac{1}{abc}\right)-(1-\delta-\varepsilon)\log r\\
    \end{split}
    \end{equation*}
for all $r\not\in E$, where $E$ has finite logarithmic measure
and $\delta=\max\{\rho(a),\rho(b),\rho(c)\}$. In addition, since
    $$
    \max\{|a|,|b|\}\leq\max\{|a|,|b|,|c|\}\leq 2\max\{|a|,|b|\},
    $$
we have
    $$
    \widetilde{T}_{a,b,c}(r)=\widetilde{T}_{a,b}(r)+O(1).
    $$
Therefore, it yields our assertion.
\end{proof}

\begin{remark}
The condition ``order less than 1" in Theorem~\ref{functionABC.the} is necessary. For example, let $a=\sin\pi z$, $b=\sin\pi(z-1/2)$, and $c=\sqrt{2}\sin\pi(z-1/4)$. It is obvious that
    $$
    \widetilde{T}_{a,b,c}(r)= r+O(1)\quad\text{and}\quad \overline{N}_\Delta\left(r,\frac{1}{abc}\right)=O(\log r).
    $$
It shows that the inequality in Theorem~\ref{functionABC.the} does not hold.
\end{remark}

The following theorem extends Theorem \ref{functionABC.the} for $m+1$ entire functions of order less than 1, see \cite[Theorem~4.1]{IW2022} for the polynomial case.

\begin{theorem}\label{m+1.theroem}
Let $m$ be an integer such that $m>2$. Suppose that $f_1,\dots,f_{m+1}$ are pairwise relatively shifting prime entire functions of order $<1$ and $f_1,\dots,f_m$ are linearly independent over the filed of periodic functions with period one satisfying
    $$
    f_1+f_2+\cdots+f_m=f_{m+1}.
    $$
Then there exists $R>0$ such that for any $\varepsilon>0$
	\begin{align*}
    \widetilde{T}_{f_1,\cdots,f_{m+1}}(r)&\leq (m-1)\overline{N}_\Delta\left(r,\frac{1}{f_1\cdots f_{m+1}}\right)-\frac{m(m-1)}{2}(1-\delta-\varepsilon)\log r\\
    &=(m-1)\sum_{j=1}^{m+1}\overline{N}_\Delta\left(r,\frac{1}{f_j}\right)-\frac{m(m-1)}{2}(1-\delta-\varepsilon)\log r
    \end{align*}
holds for $r>R$ outside an exceptional set on $r$ with finite logarithmic measure, where $\delta=\max_{1\leq j\leq m+1}\{\rho(f_j)\}$.
\end{theorem}

\begin{proof}
Suppose that entire functions $f_1,f_2,\ldots,f_{m+1}$ have no common zeros.
Since $f_1,\ldots,f_{m+1}$ are linearly independent over the filed of periodic functions with period one, it is obvious that $f_1,\ldots,f_{m}$ are linearly independent over $\C$. For each $z\in \C$, we arrange the moduli of the function values $|f_j(z)|$ in weakly decreasing order, that is,
    $$
    \abs[f_{l_1}(z)]\leq \abs[f_{l_2}(z)]\leq \cdots \leq \abs[f_{l_{m+1}}(z)],
    $$
where the integers $l_1,l_2,\ldots,l_{m+1}$ depend on $z$. According to \cite{Cartan} or \cite[Lemma 8.2]{GH2004}, there exists a positive constant $A$ that does not depend on $z$, such that
    $$
    \abs[f_j(z)]\leq A\abs[f_{l_{m+1}}(z)]\quad\text{whenever}\quad 1\leq j\leq m+1,
    $$
and $f_{l_{m+1}}$ does not vanish at $z$. Using the fact that $f_1+f_2+\cdots+f_m=f_{m+1}$, we have
    \begin{equation}\label{T.eq}
    \widetilde{T}_{f_1,\ldots,f_{m+1}}(r)    \leq\frac{1}{2\pi}\int_{0}^{2\pi}\log \abs[f_{l_{m+1}}(re^{i\theta})]\,d\theta+O(1),
    \end{equation}
where $l_{m+1}$ depends on $z$.
 Moreover, $f_{l_1},\dots,f_{l_{m+1}}$ are linear combinations of $f_1,\dots,f_m$. Then there exists a constant $K$ that depends on $z$ with upper bound $M$, such that
    $$
    \mathcal{C}(f_1,\dots,f_m)=K\mathcal{C}(f_{l_1},\dots,f_{l_m}),
    $$
where ${\mathcal C}(f_1,f_2,\dots,f_m)(z)$ is Casorati determinant, see e.g., \cite[Pages 354--357]{M-Thomson1933},~\cite[Pages 276--281]{Norlund1924}, defined by
\begin{eqnarray*}
{\mathcal C}(z)={\mathcal C}(f_1,f_2,\dots,f_m)(z)=
\left|
\begin{array}{cccc}
f_1(z) & f_2(z) & \cdots & f_m(z) \\[5pt]
f_1(z+1) &  f_2(z+1) & \cdots &  f_m(z+1) \\[5pt]
 \vdots & \vdots& \ddots & \vdots \\[5pt]
f_1(z+m-1) & f_2(z+m-1) & \cdots & f_m(z+m-1)
\end{array}
\right|.
\end{eqnarray*}
Denote a non-constant meromorphic function by
    $$
    G=\frac{\mathcal{C}(f_1,\dots,f_m)}{f_1\cdots f_{m+1}}.
    $$
Therefore, we have
	\begin{equation}\label{C.eq}
    G=\frac{K\mathcal{C}(f_{l_1},\dots,f_{l_m})}{f_{l_1}\cdots f_{l_{m+1}}}
	=\frac{K}{f_{l_{m+1}}}
    \begin{vmatrix}
    1 & 1 & \cdots & 1 \\
    \frac{f_{l_1}(z+1)}{f_{l_1}(z)} & \frac{f_{l_2}(z+1)}{f_{l_2}(z)} & \cdots & \frac{f_{l_{m}}(z+1)}{f_{l_{m}}(z)} \\
    \vdots & \vdots &  \ddots & \vdots \\
    \frac{f_{l_1}(z+m-1)}{f_{l_1}(z)} & \frac{f_{l_2}(z+m-1)}{f_{l_2}(z)} & \cdots & \frac{f_{l_m}(z+m-1)}{f_{l_m}(z)}
    \end{vmatrix}.
	\end{equation}
Keeping in mind that $f_1,\ldots,f_{m+1}$ are relatively shifting prime entire functions of order less than 1, from Lemma \ref{common.lemma} and proof of \eqref{overlinec.eq}, we have
    \begin{equation}\label{Nm-1.eq}
    N(r,G)\leq (m-1)\sum_{j=1}^{m+1}\overline{N}_\Delta\left(r,\frac{1}{f_j}\right)+O(1)
    = (m-1)\overline{N}_\Delta\left(r,\frac{1}{f_1\cdots f_{m+1}}\right)+O(1).
    \end{equation}
Hence, by Jensen formula, \cite[Theorem~5.1]{CF2009}, and \eqref{T.eq}-\eqref{Nm-1.eq}, it yields that for any $\varepsilon>0$
	\begin{equation}\label{Tf.eq}
    \begin{split}
	\widetilde{T}_{f_1,\ldots,f_{m+1}}(r)
&\leq\frac{1}{2\pi}\int_{0}^{2\pi}\log \abs[f_{l_{m+1}}]\,d\theta+O(1)\\
		&\leq N(r,G)-N\left(r,\frac{1}{G}\right)-\frac{m(m-1)}{2}(1-\delta-\varepsilon)\log r\\
		&\leq (m-1)\overline{N}_\Delta\left(r,\frac{1}{f_1\cdots f_{m+1}}\right)-\frac{m(m-1)}{2}(1-\delta-\varepsilon)\log r
	\end{split}
    \end{equation}
holds for $r>R$ outside an exceptional set on $r$ with finite logarithmic measure, where $\delta=\max_{1\leq j\leq m+1}\{\rho(f_j)\}$.

Suppose that entire functions $f_1,f_2,\ldots,f_{m+1}$ have common zeros. From Remark~\ref{commonzero.re}, there exists an entire function $h$ of order $\rho(h)<1$ with simple length zeros such that $g_j=f_j/h$ are entire without common zeros for $j=1,\ldots,m+1$. Substituting $f_j$ to \eqref{Tf.eq} and using Jensen's formula, we have
    \begin{align*}
    \widetilde{T}_{f_1,\ldots,f_{m+1}}(r)&=\widetilde{T}_{g_1,\ldots,g_{m+1}}(r)+\widetilde{T}_{h}(r)\\
    &\leq (m-1)\overline{N}_\Delta\left(r,\frac{1}{g_1\cdots g_{m+1}}\right)-\frac{m(m-1)}{2}(1-\delta'-\varepsilon)\log r+N\left(r,\frac{1}{h}\right)\\
    &=(m-1)\overline{N}_\Delta\left(r,\frac{1}{g_1\cdots g_{m+1}}\right)-\frac{m(m-1)}{2}(1-\delta'-\varepsilon)\log r+\overline{N}_\Delta\left(r,\frac{1}{h}\right)\\
    &\leq (m-1)\overline{N}_\Delta\left(r,\frac{1}{f_1\cdots f_{m+1}}\right)-\frac{m(m-1)}{2}(1-\delta-\varepsilon)\log r
    \end{align*}
holds for $r>R$ outside an exceptional set on $r$ with finite logarithmic measure, where $\delta'=\max_{1\leq j\leq m+1}\{\rho(g_j)\}\leq\delta$. Therefore, we prove the assertion.
\end{proof}

\section{Difference analogue of truncated version of Nevanlinna second main theorem}

Nevanlinna second main theorem which is a deep generalization of Picard's theorem implies that a non-constant meromorphic function cannot have too many points with high multiplicity. In this section, we proceed to show that a subnormal meromorphic function such that $\Delta f(z)\not\equiv 0$ cannot have too many points with long length. The following theorems in this section heavily depend on the difference analogue of the lemma on the logarithmic derivative. The initial assumption is finite order meromorphic functions, see \cite{CF2008} and \cite{HK2006}. The best condition so far is for meromorphic functions $f$ satisfying
    \begin{equation}\label{ZK.eq}
    \limsup_{r\to\infty}\frac{\log T(r,f)}{r}=0,
    \end{equation}
see \cite{ZK}. If a meromorphic function $f$ satisfies \eqref{ZK.eq}, then $f$ is called \emph{subnormal}.
Now let us state the difference analogue of truncated version of Nevanlinna second main theorem.

\begin{theorem}\label{Second.theorem}
Suppose that $f$ is a subnormal meromorphic function such that $\Delta f\not\equiv 0$. Let $q\geq 2$, and let $a_1,\ldots,a_q$ be distinct constants. Then
	\begin{equation}\label{diff-SMT}
		(q-1)T(r,f)\leq \overline{N}_\Delta(r,f)+\sum_{k=1}^q\overline{N}_\Delta\left(r,\frac{1}{f-a_k}\right)+S(r,f),
	\end{equation}
where the exceptional set $E$ associated with $S(r,f)$ is of zero upper density. An upper density of a set $E$ in $[1,\infty)$ is defined as
    $$
    \overline{\text{dens}}E:=\limsup_{r\to\infty}\frac{\int_{E\cup[1,r]}\,dt}{r}.
    $$
\end{theorem}

\begin{proof} In order to prove Theorem \ref{Second.theorem}, let us recall
the difference analogue of Nevanlinna second main theorem at first.

\begin{lemma}\cite{HK2006, ZK}\label{HK.lemma}
Let $c\in\C$, and let $f$ be a subnormal meromorphic function such that $\Delta f\not\equiv 0$. Let $q\geq 2$, and
let $a_1,\ldots,a_q$ be distinct constants. Then
	$$
		m(r,f)+\sum_{k=1}^qm\left(r,\frac{1}{f-a_k}\right)\leq 2T(r,f)-N_{\text{pair}}(r,f)+S(r,f),
	$$
where
	$$
	N_{\text{pair}}(r,f):=2N(r,f)-N(r,\Delta f)+N\left(r,\frac{1}{\Delta f}\right)
	$$
and the exceptional set associated with $S(r,f)$ is of zero upper density.
\end{lemma}
It is known that by Theorem \ref{pole.theorem}
    $$
    N(r,\Delta f)\leq N(r,f)+\overline{N}_\Delta(r,f)
    $$
Therefore, our result follows from Lemma \ref{HK.lemma}, \eqref{aD.eq} and Nevanlinna first main theorem.
\end{proof}

\begin{remark}
The difference analogue of Cartan's version of Nevanlinna second main theorem was given by Ishizaki et al. in \cite[Theorem~5.7]{IKLT}. It is known that Nevanlinna second main theorem
is a particular example of Cartan's version of Nevanlinna second main theorem, see \cite[Page 447]{GH2004}. The authors cannot give any proof that Theorem \ref{Second.theorem} is a corollary of
\cite[Theorem~5.7]{IKLT}.
\end{remark}

Suppose that $z_n$ are initial shifting $a$-points of $f$ with length $m_n\geq 2$. Then $z=a$ is called \emph{complete long value} of $f$. For example, let $f=z^{\underline{2}}$, $0$ is a complete long value of $f$, but for any $z\neq 0$ is not. Therefore, we get the following result from Theorem \ref{Second.theorem}.

\begin{theorem}
If $f$ is a subnormal meromorphic function such that $\Delta f\not\equiv 0$, then $f$ has at most four complete long values. In particular, if $f$ is a subnormal entire function such that $\Delta f\not\equiv 0$, then $f$ has at most two complete long values.
\end{theorem}

\begin{proof}
Suppose that $a_n$ are complete long values of a subnormal meromorphic function $f$, where $n\in\N^+$. It is known that there exists $R>0$ such that
    $$
    \overline{N}_\Delta\left(r,\frac{1}{f-a_n}\right)\leq \frac{1}{2}
    N\left(r,\frac{1}{f-a_n}\right)\leq\frac{1}{2}
    T\left(r,f\right)+O(1)
    $$
for $r>R$, where $n\in\N^+$. In addition, using Theorem \ref{Second.theorem}, we have
    \begin{equation*}
    \begin{split}
    (n-1)T(r,f)&\leq \overline{N}_\Delta(r,f)+\sum_{k=1}^n\overline{N}_\Delta\left(r,\frac{1}{f-a_k}\right)+S(r,f)\\
    &\leq T(r,f)+\frac{n}{2}T(r,f)+S(r,f),
    \end{split}
    \end{equation*}
where the exceptional set associated with $S(r,f)$ is of zero upper density. Hence, we have $n\leq 4$.
In addition, if $f$ is a subnormal entire function such that $\Delta f\not\equiv 0$, then
    $$
    (n-1)T(r,f)\leq \sum_{k=1}^n\overline{N}_\Delta\left(r,\frac{1}{f-a_k}\right)+S(r,f)
    \leq \frac{n}{2}T(r,f)+S(r,f),
    $$
where the exceptional set associated with $S(r,f)$ is of zero upper density. Hence, we have $n\leq 2$. We prove our assertion.
\end{proof}

Nevanlinna value distribution theory is concerned with the density of
points where a meromorphic function takes a certain value in the complex plane, such as the quantity $\theta(a,f)$ is the index of multiplicity of value $a$. A difference analogue of the index of multiplicity $\theta(a,f)$ is called \emph{index of height of value $a$}, which is defined as
    $$
    \theta_\Delta(a,f):=\liminf_{r\to\infty}\frac{N\left(r,\frac{1}{f-a}\right)
    -\overline{N}_\Delta\left(r,\frac{1}{f-a}\right)}{T(r,f)},
    $$
where $a\in\C$ or $a=\infty$. Obviously, $\theta_\Delta(a,f)$ is positive only if there are relatively many $a$-point of a subnormal meromorphic function $f$ with long length.

The following corollary reveals that a subnormal meromorphic function such that $\Delta f(z)\not\equiv 0$ cannot have too many $a$-points with long length.

\begin{corollary}
Let $f$ be a subnormal meromorphic function such that $\Delta f\not\equiv 0$. Then we have $\theta_\Delta(a,f)=0$ except for at most countably many values.
\end{corollary}

It is well known that any polynomial is determined by its zero except for a constant factor, but it is not true for transcendental
entire or meromorphic functions. Nevanlinna's five value theorem which is an important application of Nevanlinna second main theorem, says that if two non-constant meromorphic functions share five values ignoring multiplicity then these functions
must be identical. By considering periodic functions instead of constants, and by ignoring length instead of multiplicity, we obtain a difference analogue of the five value theorem as follows.

We say $f$ and $g$ \emph{shifting share $a$ point} in the closed disc $\overline{D}(0,r)$, if $f$ and $g$ have the same initial shifting $a$-point in the closed disc $\overline{D}(0,r)$. We note that it is possible that $a=\infty$. For example, let $f=z^{\underline{3}}=z(z-1)(z-2)$ and $g=z$. Then $f$ and $g$ shifting share $0$ in $\overline{D}(0,r)$ for $r>3$. While $h_1=z$ and $h_2=z^2$ do not shifting share $0$ in any disc. Because the initial shifting zero of $h_1$ is 0 with multiplicity 2 and of $h_2$ with multiplicity 1. If $f$ and $g$ shifting share $a$-point in $\overline{D}(0,r)$ for any $r>0$, then we say $f$ and $g$ shifting share $a$-point in the complex plane.

\begin{theorem}
Let $f$ and $g$ be subnormal meromorphic functions such that $\Delta f(z)\not\equiv 0$. If there are five distinct values $a_k$ in the extended complex plane, such that $f$ and $g$ shifting share $a_k$ in the complex plane, for $k=1,\ldots,5$. Then $f\equiv g$.
\end{theorem}

\begin{proof}
Suppose that $f\not\equiv g$. Then by Theorem \ref{Second.theorem},
    $$
    3T(r,f)\leq \sum_{k=1}^5\overline{N}_\Delta\left(r,\frac{1}{f-a_k}\right)+S(r,f)
    $$
and
    $$
    3T(r,g)\leq \sum_{k=1}^5\overline{N}_\Delta\left(r,\frac{1}{g-a_k}\right)+S(r,g)
    $$
where the exceptional set associated with $S(r,f)$ is of zero upper density.
Since $f$ and $g$ shifting share $a_k$ in the complex plane for all $k=1,2,\ldots,5$, it implies that
    $$
    \overline{N}_\Delta\left(r,\frac{1}{f-a_k}\right)=\overline{N}_\Delta\left(r,\frac{1}{g-a_k}\right)
    $$
for all $k=1,2,\ldots,5$. Therefore,
    \begin{align*}
    3(T(r,f)+T(r,g))&\leq \sum_{k=1}^5\overline{N}_\Delta\left(r,\frac{1}{f-a_k}\right)+
    \sum_{k=1}^5\overline{N}_\Delta\left(r,\frac{1}{g-a_k}\right)+
    S(r,f)+S(r,g)\\
    &\leq 2N\left(r,\frac{1}{f-g}\right)+S(r,f)+S(r,g)\\
    &\leq 2T\left(r,\frac{1}{f-g}\right)+S(r,f)+S(r,g)\\
    &\leq 2\left(T(r,f)+T(r,g)\right)+S(r,f)+S(r,g)
    \end{align*}
From inequalities above, we have
    $$
    T(r,f)+T(r,g)\leq S(r,f)+S(r,g),
    $$
which is a contradiction. There $f=g+\alpha$, where $\alpha\in\C$. Our assertion $f\equiv g$ follows by $f$ and $g$ have at least one common initial shifting $a_j$-point for $j=1,\ldots,5$.
\end{proof}

\begin{remark}
The condition $k=5$ is sharp. For example, $f=e^z$ and $f=e^{-z}$ shifting share $0,1,-1,\infty$ in the complex plane. However, $f\not\equiv g$.
\end{remark}

\section{Entire solutions of difference Fermat functional equations}

In this section, we apply Theorem \ref{functionABC.the} and Theorem \ref{m+1.theroem} to difference Fermat
type functional equations for investigating nonexistence of entire solutions of order less than 1. We adopt the falling expression $f^{\underline{n}}:=f(z)f(z-1)\ldots f(z-n+1)$ instead of $f^n$ for an entire function $f$ in the Fermat type functional equations.
For the Fermat type functional equations, see e.g., \cite{GH2004}.
Concerning the methods of the proofs, we follow the idea in~\cite{IKLT} and \cite{IW2022}.

\begin{theorem}\label{Fermat.theorem}

Let $a$, $b$ and $c$ be entire functions of order less than 1, not all constants, and $n\in\N^+$ such
that $a^{\underline{n}}$, $b^{\underline{n}}$ and $c^{\underline{n}}$ are relatively shifting prime and satisfy
    $$
    a^{\underline{n}}+b^{\underline{n}}=c^{\underline{n}}.
    $$
Then $n\leq 2$. In addition, if one of $a$, $b$ and $c$ is a constant, then $n=1$.

\end{theorem}

\begin{proof}
Let us first assume that $a$, $b$ and $c$ are all nonconstant entire functions of order less than 1. By Jensen formula, we have
    $$
    \widetilde{T}_{a^{\underline{n}}}(r)=\widetilde{T}_{a}(r)
    +\cdots+\widetilde{T}_{a(z-n+1)}(r)=N\left(r,\frac{1}{a}\right)+\cdots+
    N\left(r,\frac{1}{a(z-n+1)}\right)+O(1).
    $$
Suppose that $\{z_n\}$ are the zeros of $a$. Since there exists $R>0$ such that
    $$
    N\left(r,\frac{1}{a}\right)=\sum_{\substack{0<|a_j|\leq r\\ a_j\neq -1}}\log\frac{r}{|a_j|}+n\left(0,\frac{1}{a}\right)\log r+n\left(0,\frac{1}{a(z+1)}\right)\log r
    $$
and
    $$
    N\left(r,\frac{1}{a(z+1)}\right)=\sum_{\substack{0<|a_j-1|\leq r\\ a_j\neq 0}}\log\frac{r}{|a_j-1|}+n\left(0,\frac{1}{a(z+1)}\right)\log r+n\left(0,\frac{1}{a}\right)\log r
    $$
holds for $r>R$, by using the inequality $\log(1+x)\leq x$ for $x\geq 0$ we have
    \begin{equation*}
    \begin{split}
    N\left(r,\frac{1}{a(z+1)}\right)-N\left(r,\frac{1}{a}\right)
    &=\sum_{\substack{0<|a_j|\leq r\\0<|a_j-1|\leq r}}\left(\log\frac{r}{|a_j-1|}-\log\frac{r}{|a_j|}\right)\\
    &+\sum_{\substack{0<|a_j-1|\leq r\\|a_j|> r}}\log\frac{r}{|a_j-1|}
    -\sum_{\substack{0<|a_j|\leq r\\|a_j-1|> r}}\log\frac{r}{|a_j|}\\
    &=\sum_{\substack{0<|a_j|\leq r\\0<|a_j-1|\leq r}}\left(\log\frac{|a_j|}{|a_j-1|}\right)+O(1)\\
    &\leq\sum_{0<|a_j-1|\leq r}\frac{1}{|a_j-1|}+O(1)=O(1)
    \end{split}
    \end{equation*}
for $r>R$ when $\rho(a)<1$. Moreover, we have
    \begin{equation*}
    \begin{split}
    \overline{n}_\Delta\left(r,\frac{1}{a^{\underline{n}}}\right)&=\sum_{w\in\overline{D}(0,r)}
    \left({\rm ord}_w(a^{\underline{n}})-\min\{{\rm ord}_w(a^{\underline{n}}),{\rm Ord}_{w+1\in\overline{D}(0,r)}(a^{\underline{n}})\}\right)\\
    &=\sum_{w\in\overline{D}(0,r)}
    \left({\rm ord}_w(a^{\underline{n}})-\min\{{\rm ord}_w(a^{\underline{n}}),{\rm Ord}_{w+1\in\overline{D}(0,r)}(a^{\underline{n}})\}\right)\\
        &=\sum_{w\in\overline{D}(0,r)}
    \left(\sum_{j=0}^{n-1}{\rm ord}_{w-j}(a)-
    \min\left\{\sum_{j=0}^{n-1}{\rm ord}_{w-j}(a),\sum_{j=1}^n{\rm Ord}_{w-j\in\overline{D}(0,r)}(a)\right\}\right)\\
    &\leq \sum_{\substack{w\in\overline{D}(0,r)\\w-k\in\overline{D}(0,r) \\k\in\{1,\ldots,n\}}}
    {\rm ord}_{w}(a)
    +\sum_{\substack{w-k+1\in\overline{D}(0,r)\\w-k\not\in\overline{D}(0,r) \\k\in\{1,\ldots,n\}}}
    {\rm ord}_{w}(a)\\
    &\leq n\left(r,\frac{1}{a}\right)+\sum_{\substack{|w|\leq r\\w-k+1\in\overline{D}(0,r)\\w-k\not\in\overline{D}(0,r) \\k\in\{2,\ldots,n\}}}
    {\rm ord}_{w}(a).
    \end{split}
    \end{equation*}
It follows by Riemann-Stieltjes integral that if $\rho(a)<1$, then
    \begin{equation}\label{nN.eq}
    \overline{N}_\Delta\left(r,\frac{1}{a^{\underline{n}}}\right)\leq N\left(r,\frac{1}{a}\right)+O(1)
    \end{equation}
Hence, by \eqref{nN.eq} and Theorem~\ref{functionABC.the}, it yields that for any $\varepsilon>0$
    \begin{equation}\label{a.eq}
    \begin{split}
    nN\left(r,\frac{1}{a}\right)&=\widetilde{T}_{a^{\underline{n}}}(r)+O(1)\leq
    \widetilde{T}_{a^{\underline{n}},b^{\underline{n}},c^{\underline{n}}}(r)-(1-\delta-\varepsilon)\log r+O(1)\\
    &\leq \overline{N}_\Delta\left(r,\frac{1}{a^{\underline{n}}b^{\underline{n}}c^{\underline{n}}}\right)
    -(1-\delta-\varepsilon)\log r\\
    &=\overline{N}_\Delta\left(r,\frac{1}{a^{\underline{n}}}\right)
    +\overline{N}_\Delta\left(r,\frac{1}{b^{\underline{n}}}\right)
    +\overline{N}_\Delta\left(r,\frac{1}{c^{\underline{n}}}\right)
    -(1-\delta-\varepsilon)\log r\\
    &\leq N\left(r,\frac{1}{a}\right)+N\left(r,\frac{1}{b}\right)+N\left(r,\frac{1}{c}\right)-(1-\delta-\varepsilon)\log r
    \end{split}
    \end{equation}
holds for $r>R$ outside an exceptional set on $r$ with finite logarithmic measure, where $\delta=\max\{\rho(a),\rho(b),\rho(c)\}$.
Similarly, for $b$ and $c$, we have for any $\varepsilon>0$
    \begin{equation}\label{b.eq}
    nN\left(r,\frac{1}{b}\right)\leq N\left(r,\frac{1}{a}\right)+N\left(r,\frac{1}{b}\right)+N\left(r,\frac{1}{c}\right)-(1-\delta-\varepsilon)\log r
    \end{equation}
and
    \begin{equation}\label{c.eq}
    nN\left(r,\frac{1}{c}\right)\leq N\left(r,\frac{1}{a}\right)+N\left(r,\frac{1}{b}\right)+N\left(r,\frac{1}{c}\right)-(1-\delta-\varepsilon)\log r
    \end{equation}
hold for $r>R$ outside an exceptional set on $r$ with finite logarithmic measure.
Combining inequalities \eqref{a.eq}, \eqref{b.eq} and \eqref{c.eq}, we obtain
there exists $R>0$ such that for any $\varepsilon>0$
    $$
    (n-3)\left(N\left(r,\frac{1}{a}\right)+N\left(r,\frac{1}{b}\right)+N\left(r,\frac{1}{c}\right)\right)\leq -3(1-\delta-\varepsilon)\log r
    $$
holds for $r>R$ outside an exceptional set on $r$ with finite logarithmic measure, which implies that $n\leq 2$.

We secondly assume one of $a$, $b$ and $c$ is a constant. Without loss of generality, we assume that $c$ is a constant. Then \eqref{a.eq} and \eqref{b.eq} yield there exists $R>0$ such that for any $\varepsilon>0$
    $$
    (n-2)\left(N\left(r,\frac{1}{a}\right)+N\left(r,\frac{1}{b}\right)\right)\leq-2(1-\delta-\varepsilon)\log r
    $$
holds for $r>R$ outside an exceptional set on $r$ with finite logarithmic measure,
which shows that $n\leq 1$. We proved our assertion.
\end{proof}

The next result extends Theorem \ref{Fermat.theorem} to equations with arbitrarily many terms.

\begin{theorem}\label{Equn.eq}
Let $m\in \N$, $m\geq 2$, $n\in\N$. Suppose that there exist $f_1,\ldots,f_{m+1}$ nonconstant entire functions of order less than 1 satisfying
\begin{equation*}
    f_1^{\underline{n}}+f_2^{\underline{n}}+\cdots
    +f_m^{\underline{n}}=f_{m+1}^{\underline{n}}.
\end{equation*}
Further suppose that $f_1^{\underline{n}}$, $\ldots$, $f_{m+1}^{\underline{n}}$ are pairwise relatively shifting prime and $f_1^{\underline{n}},\ldots,f_{m}^{\underline{n}}$ are linearly independent.
Then
$$
n\leq m^2-2.
$$
\end{theorem}

\begin{proof}
By using Theorem \ref{m+1.theroem}, we have our assertion.
\end{proof}

%

\bigskip

\medskip
\noindent
\emph{R.-C. Chen}\\
\textsc{Shantou University, Department of Mathematics,\\
Daxue Road No.~243, Shantou 515063, China}\\
\texttt{email:21rcchen@stu.edu.cn}

\medskip
\noindent
\emph{Z.-T.~Wen}\\
\textsc{Shantou University, Department of Mathematics,\\
Daxue Road No.~243, Shantou 515063, China}\\
\texttt{e-mail:zhtwen@stu.edu.cn}

\vspace{1cm}
\end{document}